\documentclass[12pt]{article}
\usepackage{amsmath,amssymb,amsfonts,amsthm}
\usepackage{eucal,oldgerm}
\usepackage[dvips]{graphics}
\usepackage{pstricks,pst-plot,pst-node,pst-coil}
\newpsobject{showgrid}{psgrid}{subgriddiv=1,griddots=5,gridlabels=0pt}
\usepackage[arrow,matrix,tips,2cell]{xy}

\newcommand{\zz}{{\mathfrak{z}}}

\newcommand{\com}{{\mathbb C}}

\newcommand{\bC}{\mathsf{C}}

\newcommand\FF{\mathbb F}
\newcommand\ZZ{\mathsf Z}
\newcommand\OO{\mathcal O}
\newcommand{\T}{{\mathbf{T}}}

\newcommand{\C}{\mathbb{C}}
\newcommand{\Q}{\mathbb{Q}}
\newcommand{\Z}{\mathbb{Z}}

\newcommand{\cO}{\mathcal{O}}

\newcommand{\Pp}{{\mathbf{P}^1}}

\newcommand{\rarr}{\rightarrow}

\newcommand{\bA}{\mathcal{A}}
\newcommand{\bF}{\mathcal{F}}
\newcommand{\bD}{\mathcal{D}}
\newcommand{\bZ}{\mathsf{Z}}

\newcommand{\CC}{{\widehat{C}}}
\newcommand{\bbullet}{{\widehat{\bullet}}}

\newtheorem{Theorem}{Theorem}
\newtheorem{Lemma}{Lemma}

\newtheorem{Proposition}[Lemma]{Proposition}
\newtheorem{Conjecture}{Conjecture}

\begin{document}
\title{Descendent theory for stable pairs on toric
3-folds}
\author{R. Pandharipande and A. Pixton}
\date{April 2012}
\maketitle

\begin{abstract}
We prove the rationality of the descendent partition function
for stable pairs on
nonsingular toric 3-folds.
The method uses a geometric reduction of the
2- and 3-leg descendent vertices to the 1-leg case.
As a consequence, 
we prove the rationality of the 
relative stable pairs partition functions for all log Calabi-Yau
geometries of the form $(X,K3)$ where $X$ is a
nonsingular toric 3-fold.
\end{abstract}

\maketitle

\setcounter{tocdepth}{1} 
\tableofcontents


\setcounter{section}{-1}
\section{Introduction}

\subsection{Descendents}\label{dess}
Let $X$ be a nonsingular 3-fold, and let
$\beta \in H_2(X,\mathbb{Z})$ be a nonzero class. We will study here the
moduli space of stable pairs
$$[\OO_X \stackrel{s}{\rightarrow} F] \in P_n(X,\beta)$$
where $F$ is a pure sheaf supported on a Cohen-Macaulay subcurve of $X$, 
$s$ is a morphism with 0-dimensional cokernel, and
$$\chi(F)=n, \  \  \ [F]=\beta.$$
The space $P_n(X,\beta)$
carries a virtual fundamental class obtained from the 
deformation theory of complexes in
the derived category \cite{pt}.

Since $P_n(X,\beta)$ is a fine moduli space, there exists a universal sheaf
$$\FF \rightarrow X\times P_{n}(X,\beta),$$
see Section 2.3 of \cite{pt}.
For a stable pair $[\OO_X\to F]\in P_{n}(X,\beta)$, the restriction of
$\FF$
to the fiber
 $$X \times [\OO_X \to F] \subset 
X\times P_{n}(X,\beta)
$$
is canonically isomorphic to $F$.
Let
$$\pi_X\colon X\times P_{n}(X,\beta)\to X,$$
$$\pi_P\colon X\times P_{n}(X,\beta)
\to P_{n}(X,\beta)$$
 be the projections onto the first and second factors.
Since $X$ is nonsingular
and
$\FF$ is $\pi_P$-flat, $\FF$ has a finite resolution 
by locally free sheaves.
Hence, the Chern character of the universal sheaf $\FF$ on 
$X \times P_n(X,\beta)$ is well-defined.
By definition, the operation
$$
\pi_{P*}\big(\pi_X^*(\gamma)\cdot \text{ch}_{2+i}(\FF)
\cap(\pi_P^*(\ \cdot\ )\big)\colon 
H_*(P_{n}(X,\beta))\to H_*(P_{n}(X,\beta))
$$
is the action of the descendent $\tau_i(\gamma)$, where
$\gamma \in H^*(X,\Z)$.

For nonzero $\beta\in H_2(X,\Z)$ and arbitrary $\gamma_i\in H^*(X,\Z)$,
define the stable pairs invariant with descendent insertions by
\begin{eqnarray*}
\left\langle \prod_{j=1}^k \tau_{i_j}(\gamma_j)
\right\rangle_{\!n,\beta}^{\!X}&  = &
\int_{[P_{n}(X,\beta)]^{vir}}
\prod_{j=1}^k \tau_{i_j}(\gamma_j) \\
& = & 
\int_{P_n(X,\beta)} \prod_{j=1}^k \tau_{i_j}(\gamma_{j})
\Big( [P_{n}(X,\beta)]^{vir}\Big).
\end{eqnarray*}
The partition function is 
$$
\ZZ^X_{\beta}\left(   \prod_{j=1}^k \tau_{i_j}(\gamma_{j})
\right)
=\sum_{n} 
\left\langle \prod_{j=1}^k \tau_{i_j}(\gamma_{j}) 
\right\rangle_{\!n,\beta}^{\!X}q^n.
$$

Since $P_n(X,\beta)$ is empty for sufficiently negative
$n$, 
$\ZZ^X_{\beta}\big(   \prod_{j=1}^k \tau_{i_j}(\gamma_{j})
\big)$
is a Laurent series in $q$. The following conjecture was made in 
\cite{pt2}.

\begin{Conjecture}
\label{111} 
The partition function
$\ZZ_{\beta}^X\big(   \prod_{j=1}^k \tau_{i_j}(\gamma_{j})
\big)$ is the 
Laurent expansion of a rational function in $q$.
\end{Conjecture}

If only primary field insertions $\tau_0(\gamma)$ appear, 
Conjecture \ref{111} is known for
toric $X$ by \cite{moop, mpt} and for Calabi-Yau $X$ by
\cite{bridge,toda} together with \cite{joy}.
In case $X$ is a local curve, Conjecture \ref{111} has been
proven for descendent insertions $\tau_{>0}(\gamma)$ in \cite{part1}.

Let $\T$ be a 3-dimensional algebraic torus acting on a
nonsingular toric 3-fold $X$.{\footnote{$X$ need not be compact. In the
open case, the stable pairs invariants are defined by $\T$-equivariant
residues so long as the $\T$-fixed locus $X^\T\subset X$ is
compact.}} Let
$s_1,s_2,s_3 \in H^*_{\T}(\bullet)$ be the first  Chern classes
of the standard representations of the three factors of $\T$.
The $\T$-equivariant stable pairs invariants of $X$
take values in $\Q(s_1,s_2,s_3)$. Let
$$\ZZ_{\beta}^{X}\Big(   \prod_{j=1}^k \tau_{i_j}(\gamma_{j})
\Big)^\T\in \mathbb{Q}(s_1,s_2,s_3)((q))$$
be the $\T$-equivariant partition function with  $\gamma_j \in H^*_\T(X,\mathbb{Q})$.
The main result of the present paper is the proof of a stronger
$\T$-equivariant version of
 Conjecture 1 in the toric case.

\begin{Theorem}
\label{nnn} Let 
$X$ be a nonsingular toric 3-fold. The partition function
$\ZZ_{\beta}^{X}\big(   \prod_{j=1}^k \tau_{i_j}(\gamma_{j})
\big)^\T$ is the 
Laurent expansion in $q$ of a rational function in the field 
$\mathbb{Q}(q,s_1,s_2,s_3)$.
\end{Theorem}

\subsection{Capped descendent vertex}
Capped vertices were introduced in \cite{moop} to study 
the Gromov-Witten and Donaldson-Thomas theories of toric 3-folds.
By the same construction,
we define here capped stable pairs vertices with descendent
insertions. The 1-leg case was already treated in \cite{part1}.

Let $\T$ be a 3-dimensional algebraic torus, and let
$s_1,s_2,s_2 \in H^*_{\T}(\bullet)$ be first  Chern classes
of the standard representations of the three factors of $\T$.
Let $\T$ act diagonally on $\Pp\times \Pp \times \Pp$,
$$(\xi_1,\xi_2,\xi_3) \cdot ([x_1,y_1],[x_2,y_2],[x_3,y_3]) =
([x_1,\xi_1 y_1],[x_2,\xi_2 y_2],[x_3,\xi_3 y_3])\ .$$
Let $0,\infty \in \Pp$ be the points $[1,0]$ and $[0,1]$ respectively.
The tangent weights{\footnote{Our sign conventions here
follow \cite{moop} and disagree with \cite{part1}.
Since we will not require explicit vertex calculations here,
the sign conventions will not play a significant role.}} of $\T$ at the point
$$\mathsf{p}=(0,0,0) \in \Pp \times \Pp \times \Pp$$
are $s_1$, $s_2$, and $s_3$.

Let $U\subset \Pp\times \Pp \times \Pp$ be the $\T$-invariant 3-fold obtained
by removing the three $\T$-invariant lines 
$$
L_1,L_2,L_3 \subset \Pp \times \Pp \times \Pp
$$
passing through the point $(\infty,\infty,\infty)$,
$$U = \Pp \times \Pp \times \Pp \ \setminus \ \cup_{i=1}^3 L_i.$$ 
Let $D_i \subset U$
be the divisor with $i^{th}$ coordinate $\infty$.
For $i\neq j$, the divisors ${D}_i$ and $D_j$
are disjoint in $U$.

The capped descendent vertex is the stable pairs
partition function of $U$ with integrand
$$\tau_{\alpha_1}(\mathsf{p}) \ldots \tau_{\alpha_{\ell}}(\mathsf{p})$$ and 
free relative conditions
imposed at the divisors ${D}_i$. 
While the relative geometry $U/\cup_i D_i$
is not compact, the moduli spaces
 $P_n(U/\cup_i D_i,\beta)$
have compact $\T$-fixed loci.
The stable pairs invariants of $U/\cup_i D_i$ are
well-defined by $\T$-equivariant residues.
In the localization
formula for the reduced theories of 
 $U/\cup_i D_i$,
 nonzero degrees can occur {\em only} on the edges meeting
the origin $\mathsf{p}\in U$.

We
denote the capped stable pair descendent vertex by
\begin{eqnarray}\nonumber
\bC( \alpha  | \lambda,\mu,\nu ) & = & \mathsf{Z}
(U/\cup_i D_i, \prod_{i=1}^\ell \tau_{\alpha_i}(\mathsf{p})\ | \ \lambda,\mu,\nu)^\T \\
&
= &\sum_{n}  \label{njp}
\left\langle \prod_{i=1}^\ell \tau_{\alpha_i}(\mathsf{p}) 
\right\rangle_{\!n,\lambda,\mu,\nu}^{\T}q^n.
\end{eqnarray}
where the partition
 $\alpha$ specifies the descendent integrand and the partitions
$\lambda,\mu,\nu$ denote relative conditions imposed at $D_1,D_2,D_3$
in the Nakajima basis.
The curve class $\beta$ in \eqref{njp} is determined
by the relative conditions: $\beta$ is the sum of 
the three axes passing through $\mathsf{p}\in U$
with coefficients $|\lambda|$, $|\mu|$, and $|\nu|$ respectively. 
The superscript $\T$ after the bracket denotes
$\T$-equivariant integration on
$P_n(U/\cup_i D_i,\beta)$.

Since the parts of the partition $\alpha$ are positive,
our capped descendent vertices have no $\tau_0(\mathsf{p})$
insertions. For a stable pair $(F,s)$ on $X$, a direct
calculation shows
$$\text{ch}_2(F) \cap [X]= \beta\in H_2(X,\mathbb{Z}) .$$
Hence, $\tau_0(\mathsf{p})$ acts simply as the scalar
$$|\lambda|s_2s_3+ s_1|\mu|s_3+ s_1s_2|\nu|\ .$$
Resticting $\alpha$ to be a partition is therefore no loss.

If $\alpha=\emptyset$, there are no descendents and
 our capped descendent vertex reduces to
the capped vertex of \cite{moop}.
The basic $\Sigma_3$-action permuting the axes of $U$ implies
an $\Sigma_3$-symmetry of the capped descendent vertex.
The 2-leg and 1-leg vertices are the restrictions
$$\bC( \alpha  | \lambda,\mu,\emptyset ), \ \ \ \ \bC( \alpha  | 
\lambda,\emptyset,\emptyset ) $$
respectively.
For stable pairs, we always require $|\lambda|+|\mu|+|\nu|>0$.
However, we follow the conventions
$$\bC( \emptyset  \ |\ \emptyset,\emptyset,\emptyset ) = 1, \ \ \ \
\bC( \alpha \neq \emptyset\  |\ \emptyset,\emptyset,\emptyset ) = 0,
$$
for convenience in formulas.

We will prove Theorem \ref{nnn}
by a refined rationality result for the capped descendent vertex.
\begin{Theorem}
\label{ttt} 
For all partitions  $\alpha,\lambda,\mu,\nu$, the vertex
$\bC( \alpha  | \lambda,\mu,\nu )$
is the 
Laurent expansion in $q$ of a rational function in the field 
$\mathbb{Q}(q,s_1,s_2,s_3)$.
\end{Theorem}

The proof of Theorem \ref{ttt} uses two geometric constraints
to reduce the capped descendent vertex 
 $\bC(\alpha  | \lambda,\mu,\nu )$ to
the 1-leg case studied in \cite{part1}.
The first involves the $\bA_n$-surfaces as in \cite{moop}.
The second, for large partitions $\alpha$, involves
Hirzebruch surfaces and the relative/descendent
correspondence in the 1-leg case.
The final outcome is an effective computation of
$\bC(\alpha | \lambda,\mu,\nu )$
in terms of the capped 1-leg descendent vertex.

While we are interested here in the theory of stable pairs,
the geometric constraints used in the proof of Theorem \ref{ttt}
are equally valid for Gromov-Witten theory and Donaldson-Thomas
theory. In the latter theories, the constraints determine
the capped 2- and 3-leg descendent vertices in terms of
capped 1-leg descendent vertices.{\footnote{
The Donaldson-Thomas constraints also involve the 0-leg
descendent vertex.
The 0-leg descendent vertex
concerns degree 0 contributions which may be removed
in Gromov-Witten theory by requiring stable maps to have
no connected components contracted to a point.
The 0-leg descendent vertex is
absent in the theory of stable pairs by definition.}} 
However,
rationality does
{\em not} hold for the capped descendent vertices in the Gromov-Witten
or Donaldson-Thomas cases.

The Gromov-Witten, Donaldson-Thomas, and stable pairs
descendent theories are all conjectured to be equivalent
\cite{MNOP2,pt,pt2}. The geometric constraints studied
here show the differences between the three
descendent theories in the toric case should be viewed
as occuring in the 1-leg
descendent vertex.
A descendent correspondence, rooted in 1-leg geometry,
is proven for the
$\T$-equivariant Gromov-Witten and stable pairs
theories of all nonsingular toric 3-folds in \cite{part4}.
By the rationality result of Theorem \ref{nnn},
the MNOP \cite{MNOP1} variable change 
$$-q=e^{iu}$$
is well-defined for the stable pairs partition function.
Rationality plays a crucial role in the study of the 
correspondence in \cite{part4}.


\subsection{Log Calabi-Yau geometries}
Let $X$ be a nonsingular projective 3-fold and let 
$S\subset X$ be a nonsingular anti-canonical divisor
isomorphic to a $K3$ surface.
The pair $(X,S)$ is log Calabi-Yau,
$$K_X(S) = \OO_X\ . $$
The most basic example is $(\mathbf{P}^3,Q)$ where $Q\subset \mathbf{P}^3$
is a nonsingular quartic $K3$ surface.

There is a natural notion of pure counting in log Calabi-Yau
geometries $(X,S)$.
Let $\beta\in H_2(X,\mathbb{Z})$ be a curve class, and let
$$d= \int_{\beta}[S] \ \in \mathbb{Z}$$
be the intersection number.
Let $\text{Hilb}(S,d)$ denote the Hilbert scheme of
$d$ points on $S$, and
let $$\mathcal{L}_{\beta} \in H^{2d}(\text{Hilb}(S,d),
\mathbb{Q})$$ be a middle dimensional class.
The analogue of the partition function of Calabi-Yau invariants
 in the log Calabi-Yau situation is
\begin{equation}
\mathsf{Z}^{X/S}_\beta
\Big( 1  \Big|\  \mathcal{L}_{\beta} \Big) 
= \sum_{n}  \label{njpp}
\Big\langle 1 
\Big\rangle_{\!n, \mathcal{L}_\beta}^{X/S}q^n.
\end{equation}
The partition function \eqref{njpp} counts stable pairs of curve class $\beta$
with the relative condition
determined by the class $\mathcal{L}_\beta$.
As a consequence of Theorem \ref{nnn}, we obtain the following 
result.

\begin{Theorem}
\label{fff} Let 
$X$ be a nonsingular projective toric 3-fold
with an anti-canonical $K3$ section $S$. The partition function
$\mathsf{Z}_\beta^{X/S}
\Big(1  \Big|\  \mathcal{L}_{\beta} \Big)$ 
is the 
Laurent expansion of a rational function in $q$.
\end{Theorem}

The Hilbert scheme $\text{Hilb}(S,d)$ is well-known
to carry a canonical holomorphic sympectic form.
If $\mathcal{L}_\beta$ is obtained from a Lagrangian cycle,
a very natural approach to Theorem \ref{fff}, following
the sucessful arguments for rationality in Calabi-Yau cases, 
is to study Behrend
functions and wall-crossing for sheaf counting invariants 
associated to relative geometries. 
D.~ Maulik and R. P. Thomas have recently 
initiated a study of Behrend functions
for log Calabi-Yau geometries.

In fact, our proof of Theorem \ref{fff} yields the rationality of
all partition functions for the relative geometry $(X,S)$
with arbitrary descendent insertions and non-Lagrangian
boundary conditions.

\subsection{Further directions}
The techniques of \cite{mptop} allow the use of degeneration
to exchange relative conditions and descendent invariants.
Theorem \ref{fff} is an easy application of the ideas of \cite{mptop}
to the theory of stable pairs. New results including 
rationality of the stable pairs partition functions
in many non-toric settings will be established in
a sequel. 

The arguments of \cite{mptop} are for Gromov-Witten
theory where target dimension reduction plays a basic role. For stable pairs,
several aspects have to be redone without dimension reduction.

\subsection{Plan of the paper}
We start with a review of capped localization in Section \ref{cl}.
The full capped localization formula for descendents
is given in Section \ref{exxx}. The capped 2-leg descendent
vertex is studied in Section \ref{2ldv}. The first half of the
reduction to 1-leg uses results about $\bA_n$-surface geometries in
Section \ref{jj3}. The second half of the reduction uses 
constraints obtained from Hirzebruch surfaces
in Section \ref{jj4}. The equivalence between relative conditions
and stationary descendents for local curves (used before in
\cite{lcdt}) is reviewed in Section \ref{jj5}.
The proofs of Theorems \ref{nnn} and \ref{ttt} are
completed in Section \ref{3ldv} with the analysis of 
the capped 3-leg descendent vertex.
The proof of Theorem \ref{fff} for log Calabi-Yau geometries
is presented in Section \ref{lcy}.

\subsection{Acknowledgements}
Discussions with J. Bryan, D. Maulik, A. Oblomkov, A. Okounkov, and
R.~P. Thomas 
about the stable pairs vertex, self-dual obstruction theories,
and rationality
played an important role. 
The study of descendents
for 3-fold sheaf theories in \cite{MNOP2,pt2} 
motivated several 
aspects of the paper. 
 
R.P. was partially supported by NSF grants DMS-0500187
and DMS-1001154.
A.P. was supported by a NDSEG graduate fellowship.
The paper was completed in the fall of 2010
while visiting the 
Instituto Superior T\'ecnico in Lisbon where
R.P. was supported by a Marie Curie fellowship and
a grant from the Gulbenkian foundation.

\section{Capped localization}
\label{cl}

\subsection{Toric geometry}
Let $X$ be a nonsingular toric $3$-fold.
Virtual localization with respect to the
action of the full 3-dimensional torus $\T$ reduces all stable pairs
invariants
of $X$
to local contributions of the vertices and edges of the 
associated toric polytope. 
However, the standard
constituent pieces of the localization formula
yield transcendental functions.  
We will use the regrouped localization procedure 
introduced in \cite{moop} 
with capped vertex and edge contributions.  
The capped vertex and edge terms are equivalent 
building blocks for global toric calculations, 
but are much better behaved.

Let $\Delta$ denote the polytope associated to
$X$. The vertices of $\Delta$ are in bijection
with $\T$-fixed points $X^\T$.
The edges $e$ correspond
to $\T$-invariant curves $$C_e\subset X.$$ The three edges
incident to any vertex carry canonical $\T$-weights ---
the tangent weights of the torus action.

We will consider both compact and noncompact toric
varieties $X$. In the latter case, edges 
may be compact or noncompact.
Every compact edge is incident to two vertices.

\subsection{Capping}
Capped localization expresses the 
$\T$-equivariant stable pairs descendents 
of $X$ as a sum of capped descendent vertex and capped edge
data.

A half-edge $h=(e,v)$
 is a compact edge $e$ together with the
choice of an incident vertex $v$.
A partition assignment
$$h \mapsto \lambda(h)$$
to half-edges is {\em balanced} if the equality
$$|\lambda(e,v)| = |\lambda(e,v')|$$
always holds 
for the two halfs of  $e$.
For a balanced assignment, let $$|e|=|\lambda(e,v)|=|\lambda(e,v')|$$
denote the {\em edge degree}.

The outermost sum in the capped localization formula
runs over all balanced 
assignments of partitions $\lambda(h)$ to
the half-edges $h$ of $\Delta$ satisfying
\begin{equation}\label{ddfff}
\beta = \sum_e |e|\cdot \left[C_e\right] \in H_2(X,\mathbb{Z})\,.
\end{equation}
Such a partition assignment will be called a \emph{capped
marking} of
$\Delta$.
The weight of each capped marking in the localization sum for the
stable pairs descendent partition function equals the product
of three factors:
\begin{enumerate}
\item[(i)] capped descendent vertex contributions,
\item[(ii)] capped edge contributions,
\item[(iii)] gluing terms.
\end{enumerate}

Each vertex determines up to three half-edges specifying the
partitions for the capped vertex.
Each compact edge determines two half-edges specifying
the partitions of the capped edge.
The capped edge contributions (ii) 
and gluing terms (iii) here are {\em exactly} the same as
for the the capped localization formula in \cite{moop}.
Precise formulas are written in Section \ref{exxx}.

The capped localization formula
is easily derived from the standard localization formula (with
roots in \cite{GraberP,MNOP1}).
Indeed, the capped objects are obtained from the
uncapped objects by  rubber integral{\footnote{
Rubber integrals $\langle \lambda \ |\ \frac{1}{1-\psi_\infty} \ |  
\ \mu \rangle ^\sim$ 
arise in the localization formulas for relative
geometries.
See
\cite{part1} for a discussion.}} factors.
The rubber integrals cancel in pairs in capped localization
to yield standard localization.

\label{bgty}

\subsection{Formulas} \label{exxx}

The $\T$-equivariant cohomology of $X$ is generated (after localization)
by the classes of the $\T$-fixed points $X^\T \subset X$.
Let $\alpha$ be a partition with parts $\alpha_1, \ldots, \alpha_\ell$,
and let 
$$\sigma: \{1,\ldots, \ell\} \rightarrow X^\T\ .$$
Let $\mathsf{p}_{\sigma(i)} \in H^6_\T(X,\mathbb{Q})$ denote the
class of the $\T$-fixed point $\sigma(i)$.
We 
consider the 
capped localization formula 
for the $\T$-equivariant stable pairs descendent partition function 
\begin{equation}\label{fq2}
\bZ_\beta\Big(X,\prod_{i=1}^\ell \tau_{\alpha_i}(\mathsf{p}_{\sigma(i)})\Big)
=
\sum_{n}  
\left\langle \prod_{i=1}^\ell \tau_{\alpha_i}(\mathsf{p}_{\sigma(i)})
\right\rangle_{\!n,\beta}^{X}q^n.
\end{equation}

Let $\mathcal{V}$ be the set of vertices of $\Delta$ which
we identify 
with $X^\T$. 
For $v\in 
\mathcal{V}$, let
$\alpha^v$ be the collection of parts $\alpha_i$ of $\alpha$
satisfying $\sigma(i)=v$.
The partition $\alpha^v$ has size bounded by $|\alpha|$.

For $v\in \mathcal{V}$,
let $h^{v}_1, h^v_2, h^v_2$ be the associated half-edges{\footnote{
For simplicity, we assume $X$ is projective so each vertex is
incident to 3 compact edges.}}  
with tangent weights $s^v_1,s^v_2,s^v_3$ respectively.
Let $\Gamma_\beta$ be the set of capped markings
 satisfying the degree condition \eqref{ddfff}.
Each $\Gamma \in \Gamma_\beta$  associates  
a partition $\lambda(h)$ to every half-edge $h$. Let
$$|h| = |\lambda(h)|$$
 denote the half-edge degree.

For each $v\in \mathcal{V}$, the assignments $\sigma$ and
$\Gamma$ determines an evaluation of the capped vertex,
$$\mathsf{C}(v,\sigma,\Gamma) = \mathsf{C}(\alpha^v|\lambda(h^v_1),
\lambda(h^v_2), \lambda(h^v_3))|_{s_1= s^v_1, s_2=s^v_2, s_3=s^v_3}.$$
Let
$h^e_1$ and $h^e_2$ be the half-edges
associated to the edge $e$.
The assignment 
$\Gamma$ also determines an evaluation of the capped edge,
$$\mathsf{E}(e,\Gamma) = \mathsf{E}(\lambda(h^e_1),
 \lambda(h^e_2)).$$
The edge factors and weights are identical to the
corresponding Donaldson-Thomas edge terms in \cite{moop}. 
A gluing factor is specified by $\Gamma$ at each
half-edge $h^v_i\in \mathcal{H}$ by 
$$\mathsf{G}(h^v_i,\Gamma) = 
(-1)^{|h^v_i|-\ell(\lambda(h^v_i))}
\mathfrak{z}(\lambda(h^v_i)) 
\left(\frac{\prod_{j=1}^3 s^v_j}{s^v_i}\right)^{\ell(\lambda(h^v_i))} 
q^{-|h^v_i|}\ $$
where 
 $\zz(\lambda)$ is the order of the
centralizer in the symmetric group of
an element with cycle type $\lambda$.

The capped localization formula for stable pairs can be written
exactly in the form presented in Section \ref{bgty},
\begin{equation*}
\bZ_\beta\Big(X,\prod_{i=1}^\ell \tau_{\alpha_i}(\mathsf{p}_{\sigma(i)})\Big)=
\sum_{\Gamma\in \Gamma_\beta}\
\prod_{v\in \mathcal{V}}\ \prod_{e\in \mathcal{E}}\ \prod_{h\in \mathcal{H}}
\mathsf{C}(v,\sigma,\Gamma) \
\mathsf{E}(e,\Gamma)\
\mathsf{G}(h, \Gamma)
\end{equation*}
where the product is over the sets of vertices $\mathcal{V}$, edges
 $\mathcal{E}$, and half-edges $\mathcal{H}$ 
of the polytope $\Delta$.

The most basic example of capped localization occurs
for the 3-fold total space of
\begin{equation}\label{bb23}
\cO(a) \oplus \cO(b) \rightarrow \Pp.
\end{equation}
The standard localization formula has vertices over
$0,\infty\in \Pp$ and a single 
edge.
To write the answer in terms of capped localization,
we consider a $\T$-equivariant degeneration of  \eqref{bb23} to a chain
$$ (0,0) \cup (a,b) \cup (0,0)$$
of total spaces of bundles over $\Pp$
denoted here by splitting degrees.
The first $(0,0)$-geometry is relative over $\infty\in \Pp$,
the central $(a,b)$-geometry is relative on both sides, and
the last $(0,0)$-geometry is relative over $0\in \Pp$.
The degeneration formula exactly expresses the stable pair theory
 of \eqref{bb23} as capped
localization with 2 capped vertices and a single capped
edge in the middle.

\section{Capped descendent 2-leg vertex} \label{2ldv}

\subsection{Induction strategy}
We will prove Theorem \ref{ttt} for capped 2-leg descendent
vertices by induction.
Using the $\Sigma_3$-symmetry of the vertex, we may assume 2-leg
vertices are of the form
$$\bC(\alpha | \lambda,\mu,\emptyset), \ \ \ \  |\lambda|\geq |\mu|\geq 0 \ . $$
We know $\bC(\alpha | \lambda,\mu,\emptyset)$ is the Laurent
expansion in $q$ of a rational function if
\begin{equation}\label{tnn4}
\alpha=\emptyset\ \ \ \text{or}  \ \ \ \mu = \emptyset\ .
\end{equation}
In the former case, there are no descendents and rationality is a central result of
\cite{moop}. In the latter case, rationality for the capped 
1-leg  descendent
vertex is a central result of \cite{part1}.

Define a partial ordering on  capped 2-leg descendent
vertices by the following rules. We say
\begin{equation*}
\bC(\alpha | \lambda,\mu,\emptyset) \ \triangleright \ \bC(\alpha' | 
\lambda',\mu',\emptyset),
\end{equation*}
if we have
\begin{enumerate}
\item[$\bullet$] $|\alpha| > |\alpha'|$,
 \item[$\bullet$] or $|\alpha| = |\alpha'|$ and $|\mu| > |\mu'|$,
 \item[$\bullet$] or $|\alpha| = |\alpha'|$, $|\mu| = |\mu'|$ and
$|\lambda| > |\lambda'|$.
\end{enumerate}
The relationship $\triangleright$ is just the lexicographic
ordering on the triples $(|\alpha|,|\mu|,|\lambda|)$.

To prove Theorem \ref{ttt} in the 2-leg case for
$\bC(\alpha | \lambda,\mu,\emptyset)$, we assume
rationality holds for all vertices
$\bC(\alpha' | \lambda',\mu',\emptyset)$
occuring earlier in the partial ordering $\triangleright$.
The ground cases of the induction are \eqref{tnn4}
so we may assume $|\alpha|,|\mu|>0$.

To prove rationality for 
$\bC(\alpha | \lambda,\mu,\emptyset)$, we will use geometric
constraints. The approach depends upon 
whether $|\alpha|<|\lambda|$ or $|\alpha|\geq |\lambda|$.
In the former case, we will use $\bA_1$-surface geometry.
In the latter case, we will use Hirzebruch surfaces.

\subsection{ Case $|\alpha|<|\lambda|$}
\label{jj3}

\subsubsection{$\bA_1$ geometry} \label{angeo}
Let $\zeta$ be a primitive $(n+1)^{th}$ root
of unity, for $ n \geq 0$.
Let the generator of the
cyclic group $\mathbb{Z}_{n+1}$ act on $\com^2$ by
$$
(z_1,z_2)\mapsto  (\zeta\, z_1, \zeta^{-1}z_2)\,.
$$
Let  $\bA_n$ be  the minimal resolution of the quotient
$$
\bA_n \rightarrow \com^2/\mathbb{Z}_{n+1}.
$$
The diagonal $(\C^*)^2$-action on $\com^2$ commutes
with the action of $\mathbb{Z}_{n+1}$. As a result, the
surfaces $\bA_n$ are toric.

The surface $\bA_1$ is isomorphic to the total space of
$$\OO(-2) \rarr \Pp\ $$
and admits a toric compactification 
$$\bA_1 \subset \mathbf{P}(\OO+\OO(-2)) = \bF_2$$
by the Hirzebruch surface. 

Let $C\subset \bA_1$ be the 0-section of $\OO(-2)$, and
let $\star,\bullet\in C$ be the $(\com^*)^2$-fixed points.
Let $$\overline{\star} ,\overline{\bullet}\in\bF_2\setminus \bA_1$$
be the $(\com^*)^2$-fixed points lying above $\star,\bullet$
repectively.
We fix our $(\com^*)^2$-action by specifying tangent weights
at the four  $(\com^*)^2$-points:
\begin{eqnarray}
 T_{\star}(\bF_2):\ &  s_1-s_2, \ & \ \ 2s_2   \label{rrtt}\\
 T_{\bullet}(\bF_2):\ &  s_2-s_1, \ &\ \  2s_1 \nonumber\\
 T_{\overline\star}(\bF_2):\ &  s_1-s_2, \ & -2s_2 \nonumber\\
 T_{\overline\bullet}(\bF_2):\ &  s_2-s_1, \ & -2s_1\ . \nonumber
\end{eqnarray}
None of the tangent weights are divisible by $s_1+s_2$.

Consider the nonsingular projective toric variety  $\bF_2 \times \Pp$.
The 3-torus 
$$\T=(\com^*)^3$$
 acts on $\bF_2$ as above via the first two factors and
acts on $\Pp$ via the third factor with tangent weights $s_3$ and $-s_3$
at the points $0,\infty\in \Pp$ respectively.
The two $\T$-invariant divisors of $\bF_2 \times \Pp$
$$\bD_0 = \bF_2 \times \{0\}, \ \ \bD_\infty = \bF_2 \times \{\infty\}$$
will play a basic role.
The 3-fold $\bF_2\times \Pp$ has eight $\T$-fixed points which we denote by
$$\star_0,\overline{\star}_0,\bullet_0,\overline{\bullet}_0,
\star_\infty,\overline{\star}_\infty,\bullet_\infty,\overline{\bullet}_\infty \in 
\bF_2\times\Pp $$
where the subscript indicates the coordinate in $\Pp$.

Let $L_0 \subset \bF_2\times \Pp$ be the $\T$-invariant line
connecting $\star_0$ and $\overline{\star}_0$. 
Similarly, let $L_\infty \subset \bF_2\times \Pp$ be the $\T$-invariant line
connecting $\star_\infty$ and $\overline{\star}_\infty$. 
The lines $L_0$ and $L_\infty$
are $\Pp$-fibers of the Hirzebruch surfaces $\bD_0$ and $\bD_\infty$.
We have
$$H_2(\bF_2\times \Pp,\mathbb{Z}) = \mathbb{Z}[C] \oplus
\mathbb{Z} [L_0] \oplus
\mathbb{Z}  [P]$$
where $P$ is the fiber of the projection to $\bF_2$.

\subsubsection{Integration}
We will find relations which express
$\bC(\alpha|\lambda, \mu,\emptyset)$ in terms
of inductively treated vertices.
Let $\mu'$ be {\em any} partition. The relations will
be obtained from
vanishing stable pairs invariants of the relative
geometry $\bF_2\times \Pp/\bD_\infty$
in curve class
$$\beta = |\mu|\cdot [C] + (|\lambda|+|\mu'|)\cdot [P] \ . $$
The virtual dimension of the associated moduli space is 
$$\text{dim}^{vir}\ P_n(\bF_2\times \Pp/\bD_\infty,\beta) = 2|\lambda|+2|\mu'| \ .$$

Relative conditions in $\text{Hilb}(\bD_\infty, |\lambda|+|\mu'|)$
are best expressed in terms of the Nakajima basis given by a
$\T$-equivariant cohomology weighted partition of $|\lambda|+|\mu'|$.
We impose
the relative condition determined by the partition
$$\lambda \cup \mu'=\lambda_1+\ldots+\lambda_{\ell(\lambda')}
+\mu_1'+\ldots+\mu'_{\ell(\mu')}$$ 
weighted by
$[L_\infty]\in H^2_\T(\bD_\infty,\mathbb{Q})$ for the parts of $\lambda$
and  
$[\bullet_\infty]\in H^4_\T(\bD_\infty,\mathbb{Q})$ for the parts
of $\mu'$.
We denote the relative condition by $\mathsf{r}(\lambda,\mu')$.
After imposing $\mathsf{r}(\lambda,\mu')$,
the virtual dimension drops to
$$\text{dim}^{vir}\ P_n(\bF_2\times \Pp/\bD_\infty,\beta)_{\mathsf{r}(\lambda,\mu')} 
= |\lambda|+ |\mu'|-\ell(\mu') \geq |\lambda|\ .$$

We now specify descendent insertion. Since $|\alpha|>0$, there is a
positive part $\alpha_1$. We consider the descendent insertion
$$\tau_{\alpha_1}([L_0])\cdot \prod_{i=2}^\ell \tau_{\alpha_i}([D_0]) \ . $$
The descendent insertion imposes $|\alpha|+1$ conditions.
Therefore, the integral
\begin{equation}\label{hhhooo}
\int_{ [P_n(\bF_2\times \Pp/\bD_\infty,\beta)_{\mathsf{r}(\lambda,\mu')}]^{vir}}
\tau_{\alpha_1}([L_0])\cdot \prod_{i=2}^\ell \tau_{\alpha_i}([D_0])\ ,
\end{equation}
viewed as $\T$-equivariant push-forward to a point,
has dimension at least 
$$|\lambda|-|\alpha|-1  \geq 0.$$

\begin{Proposition} The $\T$-equivariant integral 
\eqref{hhhooo} vanishes for all $n$. \label{ggh2}
\end{Proposition}

\begin{proof} If the integral has dimension greater than 0, then
$\T$-equivariant push-forward with values in $\mathbb{Q}[s_1,s_2,s_3]$
vanishes since the moduli space 
$$P_n(\bF_2\times \Pp/\bD_\infty,\beta)_{\mathsf{r}(\lambda,\mu')}$$
is compact.

If the integral has dimension 0, then $\T$-equivariant push-forward 
is a constant in $\mathbb{Q}\subset \mathbb{Q}[s_1,s_2,s_3]$.
In particular, the integral can be computed by $\T$-equivariant
localization
followed by the specialization 
$$s_1+s_2=0.$$
The $\T$-equivariant localization formula for $\bF_2\times \Pp/\bD_\infty$
will be discussed carefully below. In fact,
very little is needed for the vanishing here.

Since the $[C]$ coefficient of $\beta$ is positive,
a $\T$-fixed stable pair
in the moduli space
must contain a component in the 
open set $\bA_1\times \Pp$ relative
to the divisor over $\infty$. The $s_1+s_2=0$ 
specialization of  localization on 
$\bA_1\times \Pp$ relative to $\infty$   is well-known to vanish 
since $\bA_1$ is
holomorphic symplectic (and the  $[C]$ coefficient is positive).
In addition, there may be components of  $\T$-fixed stable pairs with
support over $\bF_2 \setminus \bA_1$. The latter give rise to
descendent 1-leg vertex contributions which, because of the tangent
weight analysis \eqref{rrtt}, have no poles at $s_1+s_2=0$.
Hence, the substitution $s_1+s_2=0$ after localization is
well-defined and kills all contributions.
\end{proof}

\subsubsection{Relation}
We define the $\T$-equivariant series 
\begin{equation*}
\bZ_\beta\Big(\alpha,\lambda,\mu'\Big)
=
\sum_{n}  q^n
\int_{ [P_n(\bF_2\times \Pp/\bD_\infty,\beta)_{\mathsf{r}(\lambda,\mu')}]^{vir}}
\tau_{\alpha_1}([L_0])\cdot \prod_{i=2}^\ell \tau_{\alpha_i}([D_0])\ 
\end{equation*}
obtained from the integrals \eqref{hhhooo}.
By Proposition \ref{ggh2}, the series $\bZ\Big(\alpha,\lambda,\mu'\Big)_\beta$
vanishes identically.
We will calculate the left side of
\begin{equation}\label{ttx}
\bZ_\beta\Big(\alpha,\lambda,\mu'\Big)=0
\end{equation}
by capped localization to obtain a relation constraining
capped descendent vertices.

The stable pairs theory of the relative geometry $\bF_2\times \Pp/\bD_\infty$ admits a
capped localization formula. 
Over $0\in \Pp$, capped descendent vertices occur as in
the cappled localization formula of Section \ref{exxx}.
Over $\infty \in \Pp$,
capped rubber terms for $\T$-equivariant localization in the relative
geometry arise.
Capped rubber
is  discussed in Section 3.4 of \cite{moop}.
Since all our descendent
insertions lie over $0\in \Pp$, 
our capped rubber has the same definition as
the capped rubber of \cite{moop}.

Two types of capped rubber contributions arise
over $\infty\in \Pp$ in the 
$\T$-equivariant localization formula for
$\bZ_\beta\Big(\alpha,\lambda,\mu'\Big)$,
\begin{enumerate}
\item[(i)] capped 1-leg rubber corresponding to $\T$-fixed stable pairs
with components contracted to $\bF_2 \setminus \bA_1$,
\item[(ii)] capped rubber contributions of $\bA_1 \times \Pp$ relative to the
divisor over $\infty$.
\end{enumerate}
The capped contributions (i) are just 1-leg  with no
descendents, so are rational.
The capped contributions (ii) are proven to be rational in 
Lemma 6 of \cite{moop} relying on the results of \cite{mo1,mo2}.
See Section 5 of \cite{mpt} for the stable pairs results.

We now analyze the capped localization 
of $\bZ_\beta\Big(\alpha,\lambda,\mu'\Big)$
over $0\in \Pp$. 
A term in the capped localization formula is said to be
{\em principal} if not all the capped descendent
vertices which arise are known inductively to be rational.

First consider the descendent insertions.
The descendents 
$$\tau_{\alpha_1}([L_0])\cdot \prod_{i=2}^\ell \tau_{\alpha_i}([D_0]) \ $$
are free to distribute{\footnote{The classes
$[L_0],[D_0] \in H_\T^*(\bF_2\times\Pp,\mathbb{Q})$
are expressed as $\mathbb{Q}(s_1,s_2,s_3)$-linear
combinations of the $\T$-fixed points incident to $L_0$ and
$D_0$ respectively.}}
to the $\T$-fixed points over $0\in \Pp$.
By the choice of $\beta$, capped 2-leg descendent 
vertices can only occur at ${\star}_0$ and $\bullet_0$.
Descendents which distribute to $\overline{\star}_0$ and
$\overline{\bullet}_0$ will lie on capped 1-leg descendent
vertices.
The first descendent $\tau_{\alpha_1}(L_0)$ has to lie on
$\star_0$ or $\overline{\star}_0$. We conclude
all capped descendent vertices are known inductively to be rational except possibly 
when all the descendents lie on $\star_0$.

Next consider the edge degree $d$ of $C$ over $0\in \Pp$ in
the capped localization formula.
If $d<|\mu|$, then the capped descendent vertex at $\star_0$
is known inductively to be rational since the minimal relative partition is
lower. Hence, we restrict ourselves to the principal terms where $d=|\mu|$.

Since all of $|\mu|\cdot [C]$ occurs over $0\in \Pp$, the
rubber over $\infty\in\Pp$ is all 1-leg.
By the relative conditions imposed by $\mu'$ with weights
$[\bullet_\infty]$, the principal terms all have a
capped 2-leg vertex with no descendents at $\bullet_0$ with outgoing
partitions of size $|\mu|$ and $|\mu'|.$

Finally, consider the relative conditions $\lambda$ weighted
by $[L_\infty]$.
The weight allows the parts to distribute to $\overline{\star}_\infty$.
Such a distribution would result in an inductively treated
capped descendent vertex at $\star_0$ with lower maximal
relative partition.

In the principal terms of the
 capped localization of
 \eqref{ttx}, precisely the following set of
capped 2-leg descendent vertices occur at $\star_0$:
\begin{equation}\label{bxs}
\Big\{ \ \bC(\alpha|{\lambda},\widehat{\mu},\emptyset)\ \Big| \ 
|\widehat{\mu}|=|\mu| \ \Big\}.
\end{equation}
The principal terms arise as displayed in Figure \ref{A1cap}.
In addition to the vertex
$\bC(\alpha|{\lambda},\widehat{\mu},\emptyset)$ at $\star_0$,
there is a capped edge with partitions
$$|\widehat{\mu}|=|\widehat{\mu}'|$$
along the curve $C$ over $0\in \Pp$.
Finally, there is
capped 2-leg vertex with no descendents at $\bullet_0$ with outgoing
partitions $\widehat{\mu}'$ and $\mu'.$

\vspace{10pt}
\psset{unit=0.3 cm}
\begin{figure}[!htbp]
  \centering
   \begin{pspicture}(0,0)(19,20)
   \psline(4,4)(16,4)
   \psline(4,16)(16,16)
   \psline(0,0)(4,4)(4,16)(0,20)
   \pszigzag[coilarm=0.1,coilwidth=0.5,linearc=0.1](12,20)(16,16)
   \pszigzag[coilarm=0.1,coilwidth=0.5,linearc=0.1](16,4)(16,16)
   \pszigzag[coilarm=0.1,coilwidth=0.5,linearc=0.1](12,0)(16,4)
   \psline[doubleline=true](3,7)(4,8)(5,8)
   \psline[doubleline=true](3,13)(4,12)(5,12)
   \rput[l](2.3,16){$\star_0$}
   \rput[l](2.3,4){$\bullet_0$}
   \rput[l](17,18){$\lambda$}
   \rput[l](16.2,16){$\star_\infty$}
   \rput[l](16.2,4){$\bullet_\infty$}
   \rput[l](17,2){$\mu'$}
   \rput[c](6,8){$\widehat{\mu}'$}
   \rput[c](6,12){$\widehat{\mu}$}
   \end{pspicture}
 \caption{Principal terms}
  \label{A1cap}
\end{figure}
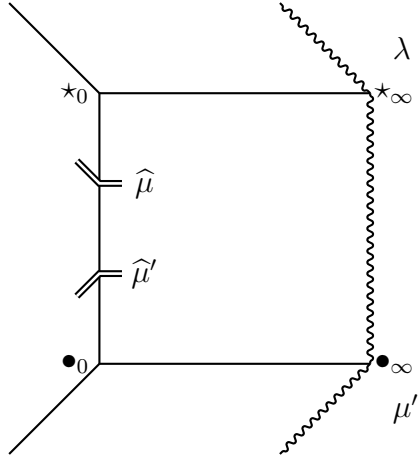

The system of equations as the partition $\mu'$ varies has
 unknowns \eqref{bxs} parameterized
by partitions of $|\mu|$. However, the number of equations
is infinite. The induction step is established if
the set of equations as $\mu'$ varies has maximal 
rank{\footnote{Maximal rank here is
equal to the number $\mathcal{P}(|\mu|)$
of partitions of size $|\mu|$.}} with respect to the unknowns \eqref{bxs}.

The maximal rank statement is proven with two observations.
First, the capped edge matrix along $C$ has maximal rank
\cite{moop}.
Second, the matrix of capped 2-leg vertices
$$\bC(\emptyset|\ \widehat{\mu}',\mu',\emptyset)$$
has maximal rank $\mathcal{P}(|\mu|)$ even when $\mu'$ varies only among
partitions of size at most $|\mu|-1$ by 
Lemma 9 of \cite{moop}.
\qed

\subsection{ Case $|\alpha|\geq |\lambda|$}
\label{jj4}

\subsubsection{Hirzebruch surfaces} \label{wrrt}
We require a different geometric construction for the
inductive relation in case $|\alpha|\geq |\lambda|$.

Let $k>0$ be an integer.
Let $\bF_k$ be the Hirzebruch surfaces given by the 
projective bundle
$$\mathbf{P}(\OO \oplus \OO(k)) \rarr \Pp\ .$$
The split presentation distinguishes
two sections
$$C^+, C^- \subset \bF_k$$
of self-intersections $k$ and $-k$ respectively.
Let $(\com^*)^2$ act on $\bF_k$ with fixed points
$$\star^+,\bullet^+, \star^-,\bullet^- \in \bF_k$$
where the first two lie on $C^+$ and the last two lie on
$C^-$. The 2-torus invariant curves of $\bF_k$ are then
\begin{equation}\label{23t}
C^+,C^-, L^\star, L^\bullet \subset \bF_k
\end{equation}
where $L^\star$ is a fiber of the projective bundle
connecting $\star^+$ and $\star^-$ and similarly for
$L^\bullet$.

Consider the nonsingular projective toric variety  $\bF_k \times \Pp$.
The 3-torus 
$$\T=(\com^*)^3$$
 acts on $\bF_k$ as above via the first two factors and
acts on $\Pp$ via the third factor with tangent weights $s_3$ and $-s_3$.
The 3-fold $\bF_k\times \Pp$ has eight $\T$-fixed points which we denote by
$$\star^+_0,\bullet^+_0, \star^-_0,\bullet^-_0,
\star^+_\infty,\bullet^+_\infty, 
\star^-_\infty,\bullet^-_\infty \in \bF_k\times \Pp $$
where the subscript indicates the coordinate in $\Pp$.

There are twelve $\T$-invariant curves of $\bF_k \times \Pp$.
There are four of type \eqref{23t} lying over $0\in \Pp$,
four lying over $\infty\in \Pp$ and
four fibers 
$$P_{\star^+},P_{\bullet^+}, P_{\star^-},P_{\bullet^-}$$
of the projection of $\bF_k\times \Pp$
to $\bF_k$.
We have
$$H_2(\bF_k\times \Pp,\mathbb{Z}) = \mathbb{Z}[C^+_0] \oplus
\mathbb{Z} [L^\star_0] \oplus
\mathbb{Z}  [P_{\star^+}].$$

\subsubsection{Integration}
We will find relations which express
$\bC(\alpha|\lambda, \mu,\emptyset)$ in terms
of inductively treated vertices.
Let $k$ be a positive integer satisfying
$$k >  3|\alpha| + 3|\lambda|+3|\mu|\ .$$
The relations will
be obtained from
vanishing stable pairs invariants of 
the nonsingular projective toric
3-fold $\bF_k\times \Pp$
in curve class
$$\beta = |\mu|\cdot [C_0^+] + |\lambda|\cdot [L_0^\star] \ . $$
The virtual dimension of the associated moduli space is 
$$\text{dim}^{vir}\ P_n(\bF_k\times\Pp, \beta) > 
3|\alpha| + 3|\lambda|+3|\mu|\ $$
since we easily compute
$$c_1(\bF_k\times \Pp) \cdot [C^+_0] = k+2\ , \ \ \ 
c_1(\bF_k\times \Pp) \cdot [L^\star_0] = 2
.$$

Our relations will be parameterized by two partitions 
$\lambda'$ and $\mu'$ 
which satisfy
$$|\lambda'|\leq|\lambda|-1, \ \ \ |\mu'|\leq |\mu|-1 \ .$$
 We consider the descendent insertion
\begin{equation}\label{nyr}
\prod_{i=1}^{\ell(\alpha)} \tau_{\alpha_i}([\star^+_0]) \cdot
\prod_{i=1}^{\ell(\lambda')} \tau_{\lambda'_i}([\star^-_0]) \cdot
\prod_{i=1}^{\ell(\mu')} \tau_{\mu'_i}([\bullet^+_0]) \ .
\end{equation}
Since $\tau_r(\mathsf{p})$ imposes $r+2$ conditions,
the descendent insertion \eqref{nyr} imposes at most
$$3|\alpha|+ 3|\lambda'|+ 3|\mu'|$$ 
conditions.

\begin{Proposition} The $\T$-equivariant integral 
\begin{equation*}
\int_{ [P_n(\bF_k\times \Pp,\beta)]^{vir}}
\prod_{i=1}^{\ell(\alpha)} \tau_{\alpha_i}([\star^+_0]) \cdot
\prod_{i=1}^{\ell(\lambda')} \tau_{\lambda'_i}([\star^-_0]) \cdot
\prod_{i=1}^{\ell(\mu')} \tau_{\mu'_i}([\bullet^+_0])
\end{equation*}
vanishes for all $n$. \label{ggh22}
\end{Proposition}

\begin{proof}
The integral,
viewed as $\T$-equivariant push-forward to a point,
has dimension at least 
$$3|\alpha| + 3|\lambda|+3|\mu| - (3|\alpha|+ 3|\lambda'|+ 3|\mu'|) >0,$$
so always vanishes.
\end{proof}

\subsubsection{Relation}
We define the $\T$-equivariant series 
\begin{equation*}
\bZ_\beta\Big(\alpha,\lambda',\mu'\Big)
=
\sum_{n}  q^n
\int_{ [P_n(\bF_k\times \Pp,\beta)]^{vir}}
\prod_{i=1}^{\ell(\alpha)} \tau_{\alpha_i}([\star^+_0]) \cdot
\prod_{i=1}^{\ell(\lambda')} \tau_{\lambda'_i}([\star^-_0]) \cdot
\prod_{i=1}^{\ell(\mu')} \tau_{\mu'_i}([\bullet^+_0])\ .
\end{equation*}
By Proposition \ref{ggh22}, 
the series $\bZ_\beta\Big(\alpha,\lambda',\mu'\Big)$
vanishes identically.
We will calculate the left side of
\begin{equation}\label{ttxx}
\bZ_\beta\Big(\alpha,\lambda',\mu'\Big)=0
\end{equation}
by capped localization to obtain a relation constraining
capped descendent vertices.

As before,
a term in the capped localization formula is 
{principal} if not all the capped descendent
vertices which arise are known inductively to be rational.
Because the curve class $\beta$ lies in the fiber of the
projection of $\bF_k\times \Pp$ to $\Pp$, no capped descendent vertex in
the localization formula will have more than 2 legs.
Capped descendent vertices at  $\T$-fixed points other than
$\star^+_0$ will have descendent partition of size less than
$|\alpha|$ and thus are known inductively to be rational.

Consider the capped 2-leg descendent vertex at $\star^+_0$.
Let $d_C$ be the associated edge degree of $C^+_0$ in
the capped localization formula.
By the geometry of $\bF_k$, the class
$$\beta - d_C[C^+_0]\in H_2(\bF_k,\mathbb{Z})$$
is not effective if $d_C>|\mu|$.
If $d_C<|\mu|$, then the vertex at $\star_0$
is known inductively to be rational since the minimal relative partition is
lower. Hence, we restrict ourselves to the principal terms where $d_C= |\mu|$.

Finally, consider the edge degree $d_L$ of $L^\star_0$.
If $d_L<|\lambda|$, then the vertex at $\star_0$
is known inductively to be rational since the maximal relative partition is
lower. Certainly $d_L>|\lambda|$ is not permitted.
We restrict ourselves to the principal terms where $d_L= |\lambda|$.

In the principal terms of the
 capped localization of
 \eqref{ttxx}, precisely the following set of
capped 2-leg descendent vertices occur at $\star^+_0$:
\begin{equation}\label{bxss}
\Big\{ \ \bC(\alpha|\widehat{\lambda},\widehat{\mu},\emptyset)\ \Big| \ 
|\widehat{\lambda}|=|\lambda|,\ 
|\widehat{\mu}|=|\mu| \ \Big\}.
\end{equation}
The principal terms arise as displayed in Figure \ref{xxcap}.
In addition to the vertex
$\bC(\alpha|\widehat{\lambda},\widehat{\mu},\emptyset)$ at $\star^+_0$,
there are capped edges along 
$C_0^+$ and $L^\star_0$ 
with partitions
$$|\widehat{\lambda}|=|\widehat{\lambda}'|,\ \ \
|\widehat{\mu}|=|\widehat{\mu}'|$$
respectively.
Finally, there are
capped 1-leg descendent vertices 
$$\bC(\lambda'|\widehat{\lambda}',\emptyset,\emptyset), \ \ \
\bC(\mu'|\widehat{\mu}',\emptyset,\emptyset)$$
at 
$\star_0^-$ and
$\bullet_0^+$.

\vspace{10pt}
\psset{unit=0.3 cm}
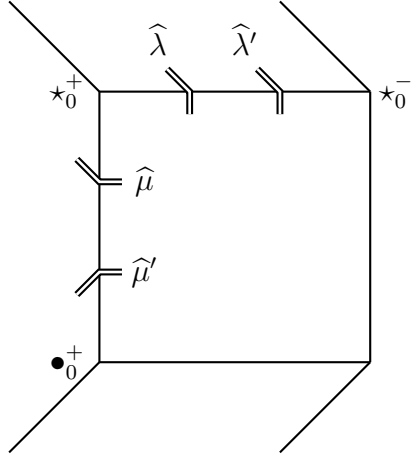
\begin{figure}[!htbp]
  \centering
   \begin{pspicture}(0,0)(19,20)
   \psline(4,4)(16,4)
   \psline(4,16)(16,16)
   \psline(0,0)(4,4)(4,16)(0,20)
   \psline(12,20)(16,16)
   \psline(16,4)(16,16)
   \psline(12,0)(16,4)
   \psline[doubleline=true](8,15)(8,16)(7,17)
    \psline[doubleline=true](12,15)(12,16)(11,17)
   \psline[doubleline=true](3,7)(4,8)(5,8)
   \psline[doubleline=true](3,13)(4,12)(5,12)
   \rput[tr](7,19){$\widehat{\lambda}$} 
\rput[tr](11,19){$\widehat{\lambda}'$}
   \rput[l](1.8,16){$\star^+_0$}
   \rput[l](1.8,4){$\bullet_0^+$}
   \rput[l](16.4,16){$\star^-_0$}
   \rput[c](6,8){$\widehat{\mu}'$}
   \rput[c](6,12){$\widehat{\mu}$}
   \end{pspicture}
 \caption{Principal terms}
  \label{xxcap}
\end{figure}

The system of equations as the partitions $\lambda'$ and $\mu'$ vary has
 unknowns \eqref{bxss} parameterized
by partitions of $|\lambda|$ and $|\mu|$.
The induction step is established if
the set of equations as $\lambda'$  and $\mu'$ vary has maximal 
rank with respect to the unknowns \eqref{bxss}.

The maximal rank statement is proven by the following observations.
The capped edge matrices along $C_0^+$ and $L^\star_0$ have maximal rank
\cite{moop}.
The matrix of capped 1-leg vertices
$$\bC(\lambda'|\ \widehat{\lambda}',\emptyset,\emptyset)$$
has maximal rank $\mathcal{P}(|\lambda|)$  when $\lambda'$ varies among
partitions of size at at most $|\lambda|-1$ by the
relative/descendent correspondence discussed in Section \ref{jj5}
below.
The matrix of capped 1-leg vertices
$$\bC(\mu'|\ \widehat{\mu}',\emptyset,\emptyset)$$
has maximal rank $\mathcal{P}(|\mu|)$  when $\mu'$ varies among
partitions of size at at most $|\mu|-1$ for the same reason.
\qed

\label{fvv3}

\subsection{ Relative/descendent correspondence}
\label{jj5}
The remaining step in the analysis of the relations in Section \ref{fvv3}
is to show the following maximal rank statement.
\begin{Proposition}\label{gbb5}
Let $d>0$ be an integer.
The matrix with coefficients
$$\bC(\alpha|\lambda,\emptyset,\emptyset)$$
as  $\alpha$ varies among
partitions of size at most $d-1$
and  $\lambda$ varies among paritions of size $d$
has maximal rank $\mathcal{P}(d)$.
\end{Proposition}

\begin{proof}
Consider the Hilbert scheme of points ${\text{Hilb}}(\com^2,d)$ of the
plane.
Let
${\mathbb{F}}$ be the universal quotient sheaf on 
$$\text{Hilb}(\com^2,d) \times \com^2,$$ and define the 
descendent{\footnote{The Chern character of $\mathbb{F}$
is properly supported over ${\text{Hilb}}(\com^2,d)$.} }
\begin{equation}\label{mrr}
\tau_c=\pi_*\Big( {\text{ch}}_{2+c}({\mathbb{F}})\Big) 
\in H^{2c}({\text{Hilb}}(\com^2,d),\mathbb{Q})
\end{equation}
where $\pi$ is the projection
$$\pi: \text{Hilb}(\com^2,d) \times \com^2 \rightarrow
\text{Hilb}(\com^2,d)\ .
$$
Define the tautological sheaf by
$$\pi_*(\mathbb{F}) =\mathcal{E}_d \rightarrow \text{Hilb}(\com^2,d) .$$
The tautological sheaf is a rank $d$ vector bundle with Chern
classes
\begin{equation}\label{jr45}
c_1, \ldots, c_{d-1} \in H^{*}({\text{Hilb}}(\com^2,d),\mathbb{Q}).
\end{equation}
Since $\mathcal{E}_d$ has a nonvanishing section, $c_d$ vanishes.

By a basic result of Ellingsrud and Str\"omme  \cite{es},   
$H^*({\text{Hilb}}(\com^2,d),\mathbb{Q})$
is generated as an algebra by the classes \eqref{jr45}. 
Hence, by Grothendieck-Riemann-Roch, 
$H^*({\text{Hilb}}(\com^2,d),\mathbb{Q})$
is also generated by 
$\tau_1, \ldots, \tau_{d-1}$.
Since 
$$H^{>2(d-1)}({\text{Hilb}}(\com^2,d),\mathbb{Q})=0,$$
graded homogeneous polynomials of degree at most $d-1$
in the $\tau_c$ span $H^*({\text{Hilb}}(\com^2,d),\mathbb{Q})$
additively.

We examine the leading $q^d$ coefficients
of the matrix coefficients,
\begin{equation}\label{lllw}
\text{Coeff}_{q^d}\Big( \bC(\alpha|\lambda,\emptyset,\emptyset)\Big) =
(s_1s_2)^\ell
\Big\langle \tau_{\alpha_1}\ldots \tau_{\alpha_\ell} \ \Big| \ \lambda 
\Big\rangle_{{\text{Hilb}}(\com^2,d)}
\end{equation}
where the bracket on the right denotes
the $(\com^*)^2$-equivariant
intersection pairing 
on the Hilbert scheme. 
On the left side of the bracket are the
classes $\tau_c$ defined above, and on the right side are
the Nakajima basis elements.
Since the $(\com^*)^2$-equivariant intersection
pairing is non-degenerate, we need only show 
graded homogeneous polynomials of degree at most $d-1$
in the classes $\tau_c$ span 
$$H^*_T({\text{Hilb}}(\com^2,d),\mathbb{Q})\otimes \Q(s_1,s_2).$$
The result follows from the non-equivariant spanning
statement since
$$\text{dim}_{\mathbb{Q}}\ H^*({\text{Hilb}}(\com^2,d),\mathbb{Q})=
\text{dim}_{\mathbb{Q}(s_1,s_2)}\ 
H^*_T({\text{Hilb}}(\com^2,d),\mathbb{Q})\otimes \Q(s_1,s_2)
=\mathcal{P}(d)\ . $$

Another approach to Proposition \ref{gbb5} is explained in
Proposition 9 of \cite{part1}.
To each partition $\gamma$ of $d$, we associate a
partition $\alpha(\gamma)$ of at most $d-1$ by removing
1 from each part of $\gamma$.
The leading $q^d$ term of the
square matrix 
$$\bC(\alpha(\gamma)|\lambda,\emptyset,\emptyset)$$
 is easily seen to be triangular when ordered
by length of partition. The
diagonal elements can be shown to not vanish by the single
calculation
\begin{equation}\label{appp}
 s_1s_2\ \Big\langle \tau_{c-1} \ \Big| \ (c) \Big\rangle _
{{\text{Hilb}}(\com^2,d)}
= \frac{1}{c!}\ 
\end{equation}
which appears in \cite{part2}.
\end{proof}

\section{Capped descendent 3-leg vertex}\label{3ldv}

\subsection{Induction strategy}
We will prove Theorem \ref{ttt} for capped 3-leg descendent
vertices by induction  by a method parallel to the 2-leg case.
We may assume 3-leg
vertices are of the form
$$\bC(\alpha | \lambda,\mu,\nu), \ \ \ \  |\lambda|\geq |\mu|\geq |\nu| \ . $$
We know $\bC(\alpha | \lambda,\mu,\nu)$ is the Laurent
expansion in $q$ of a rational function if
\begin{equation}\label{tnn5}
\alpha=\emptyset\ \ \ \text{or}  \ \ \ \nu = \emptyset\ .
\end{equation}
The latter case uses the established rationality for the capped 2-leg
descendent
vertex.

Define a partial ordering on  capped 3-leg descendent
via the
lexicographic
ordering on the vectors $(|\alpha|,|\nu|,|\mu|,|\lambda|)$.
To prove Theorem \ref{ttt} in the 3-leg case for
$\bC(\alpha | \lambda,\mu,\nu)$, we assume
rationality holds for all vertices
$\bC(\alpha' | \lambda',\mu',\nu')$
occuring earlier in the partial ordering.
The ground cases of the induction are \eqref{tnn5}
so we may assume $|\alpha|,|\nu|>0$.
Again, our strategy depends
upon
whether $|\alpha|<|\lambda|$ or $|\alpha|\geq |\lambda|$.

\subsection{ Case $|\alpha|<|\lambda|$}

\subsubsection{$\bA_2$ geometry} \label{angeo2}
Let $\bA_2\subset \bF$ 
be any nonsingular projective 
toric compactification.
We will only be interested in the 
two $(-2)$-curves of $\bA_2$,
$$C,\CC \subset \bA_2\ .$$
No other curves of $\bF$ will play a role in the
construction.

Let $\bbullet,\star,\bullet\in \bA_2$ be the $(\com^*)^2$-fixed points.
The curve $\CC$ connects $\bbullet$ to $\star$ and $C$ connects
$\star$ to $\bullet$.
The other $(\com^*)^2$-fixed points in $\bF\setminus \bA_2$
will not play an important role.

Consider the nonsingular projective toric variety  $\bF \times \Pp$.
The 3-torus 
$$\T=(\com^*)^3$$
 acts on $\bF$ via the first two factors and
acts on $\Pp$ via the third factor with tangent weights $s_3$ and $-s_3$
at the points $0,\infty\in \Pp$ respectively.
Let
$$\bD_0 = \bF \times \{0\}, \ \ \bD_\infty = \bF \times \{\infty\}$$
be $\T$-invariant divisors of $\bF_2 \times \Pp$.
The 3-fold $\bF\times \Pp$ has six important
 $\T$-fixed points which we denote by
$$\bbullet_0,\star_0,\bullet_0,
\bbullet_\infty,\star_\infty,\bullet_\infty
\in
\bF\times\Pp $$
where the subscript indicates the coordinate in $\Pp$.

Let $L_\infty \subset \bF\times \Pp$ be the $\T$-invariant line
connecting $\star_\infty$ to $(\bF\setminus \bA_2)_\infty$. 
We have
$$H_2(\bF\times \Pp,\mathbb{Z}) \supset \mathbb{Z}[C] \oplus
\mathbb{Z} [\CC] \oplus
\mathbb{Z}  [P]$$
where $P$ is the fiber of the projection to $\bF$.

\subsubsection{Integration}
We will find relations which express
$\bC(\alpha|\lambda, \mu,\nu)$ in terms
of inductively treated vertices.
Let $\mu'$ and $\nu'$ be {any} partitions. The relations will
be obtained from
vanishing stable pairs invariants of 
 $\bF\times \Pp/\bD_\infty$
in curve class
$$\beta = |\mu|\cdot [C]+ |\nu|\cdot [\CC]+ 
(|\lambda|+|\mu'|+|\nu'|)\cdot [P] \ . $$
The virtual dimension of the associated moduli space is 
$$\text{dim}^{vir}\ P_n(\bF\times \Pp/\bD_\infty,\beta) 
= 2|\lambda|+2|\mu'| +2|\nu'|\ .$$

We impose
relative conditions along $\bD_\infty$
in the Nakajima basis by the partition
$\lambda \cup \mu'\cup \nu'$
weighted by
$[L_\infty]\in H^2_\T(\bD_\infty,\mathbb{Q})$ for the parts of $\lambda$
and  
$$[\bullet_\infty],[\bbullet_\infty]\in H^4_\T(\bD_\infty,\mathbb{Q})$$
 for the parts
of $\mu'$ and $\nu'$ respectively.
We denote the relative condition by $\mathsf{r}(\lambda,\mu',\nu')$.
After imposing $\mathsf{r}(\lambda,\mu',\nu')$,
the virtual dimension drops to
$$\text{dim}^{vir}\ 
P_n(\bF\times \Pp/\bD_\infty,\beta)_{\mathsf{r}(\lambda,\mu',\nu')} 
= |\lambda|+ |\mu'|-\ell(\mu')+
|\nu'|-\ell(\nu')
 \geq |\lambda|\ .$$

We now specify the descendent insertion{\footnote{Unlike
the 2-leg case, the descendent $\tau_{\alpha_1}$
is not treated separately.}} by
$$\prod_{i=1}^\ell \tau_{\alpha_i}([D_0]) \ . $$
The descendent insertion imposes $|\alpha|$ conditions.
Therefore, the integral
\begin{equation}\label{hhhooo2}
\int_{ [P_n(\bF\times \Pp/\bD_\infty,\beta)_{\mathsf{r}(\lambda,\mu',\nu')}]^{vir}}
\prod_{i=1}^\ell \tau_{\alpha_i}([D_0])\ ,
\end{equation}
viewed as $\T$-equivariant push-forward to a point,
has dimension at least 
$$|\lambda|-|\alpha|  > 0.$$
Hence, we conclude the following vanishing.

\begin{Proposition} The $\T$-equivariant integral 
\eqref{hhhooo2} vanishes for all $n$. \label{ggh225}
\end{Proposition}

\subsubsection{Relation}
We define the $\T$-equivariant series 
\begin{equation*}
\bZ_\beta\Big(\alpha,\lambda,\mu',\nu'\Big)
=
\sum_{n}  q^n
\int_{ [P_n(\bF\times \Pp/\bD_\infty,\beta)_{\mathsf{r}(\lambda,\mu',\nu')}]^{vir}}
\prod_{i=1}^\ell \tau_{\alpha_i}([D_0])\ 
\end{equation*}
obtained from the integrals \eqref{hhhooo2}.
By Proposition \ref{ggh225}, the series 
$\bZ_\beta\Big(\alpha,\lambda,\mu',\nu'\Big)$
vanishes identically.
We will calculate the left side of
\begin{equation}\label{ttx2}
\bZ_\beta\Big(\alpha,\lambda,\mu',\nu'\Big)=0
\end{equation}
by capped localization to obtain a relation constraining
capped descendent vertices.

Two types of capped rubber contributions arise
over $\infty\in \Pp$ in the 
$\T$-equivariant localization formula for
$\bZ_\beta\Big(\alpha,\lambda,\mu',\nu'\Big)$,
\begin{enumerate}
\item[(i)] capped 1-leg rubber corresponding to $\T$-fixed stable pairs
with components contracted to $\bF \setminus \bA_2$,
\item[(ii)] capped rubber contributions of $\bA_2 \times \Pp$ relative to the
divisor over $\infty$.
\end{enumerate}
The capped contributions (i) are just  1-leg  with no
descendents, so are rational.
The capped contributions (ii) are proven to be rational in 
Lemma 6 of \cite{moop} relying on the results of \cite{mo1,mo2}.
See Section 5 of \cite{mpt} for the stable pairs results.

We now analyze the capped localization 
of $\bZ_\beta\Big(\alpha,\lambda,\mu',\nu'\Big)$
over $0\in \Pp$. As before,
a term in the capped localization formula is 
{principal} if not all the capped descendent
vertices which arise are known inductively to be rational.

First consider the descendent insertions.
The descendents 
$$\prod_{i=1}^\ell \tau_{\alpha_i}([D_0]) \ $$
are free to distribute to the $\T$-fixed points over $0\in \Pp$.
By the choice of $\beta$, a capped 3-leg descendent 
vertex can only occur at ${\star}_0$.
Descendents which distribute away from to ${\star}_0$ 
will lie on capped 1-leg or 2-leg descendent
vertices.
We conclude
all capped 
descendent vertices are known inductively to be rational except possibly 
when all the descendents lie on $\star_0$.

Next consider the edge degree $\widehat{d}$ of $\CC$ over $0\in \Pp$ in
the capped localization formula.
If $\widehat{d}<|\nu|$, then the capped descendent vertex at $\star_0$
is known inductively to be rational since the minimal relative partition is
lower. Hence, we restrict ourselves to the principal terms where 
$\widehat{d}=|\nu|$.

Similarly, consider the edge degree ${d}$ of $C$ over $0\in \Pp$ in
the capped localization formula.
If ${d}<|\mu|$, then the capped descendent vertex at $\star_0$
is known inductively to be rational since the middle relative partition is
lower. Hence, we further restrict ourselves to the principal terms where 
${d}=|\mu|$.

Since all of $|\mu|\cdot [C]+|\nu|\cdot [\CC]$ occurs over $0\in \Pp$, the
rubber over $\infty\in\Pp$ is all 1-leg.
By the relative conditions imposed by $\nu'$ with weights
$[\bbullet_\infty]$, the principal terms all have a
capped 2-leg vertex with no descendents at
 $\bbullet_0$ with outgoing 
partitions of size $|\nu|$ and $|\nu'|.$
Similarly, the principal terms all have a
capped 2-leg vertex with no descendents at
 $\bullet_0$ with outgoing 
partitions of size $|\mu|$ and $|\mu'|.$

Finally, consider the relative conditions $\lambda$ weighted
by $[L_\infty]$.
The weight allows the parts to distribute to $\bF\setminus \bA_2$.
Such a distribution would result in an inductively treated
capped descendent vertex at $\star_0$ with lower maximal
relative partition.

In the principal terms of the
 capped localization of
 \eqref{ttx2}, precisely the following set of
capped 3-leg descendent vertices occur at $\star_0$:
\begin{equation}\label{bxs2}
\Big\{ \ \bC(\alpha|{\lambda},\widehat{\mu},\widehat{\nu})\ \Big| \ 
|\widehat{\mu}|=|\mu|, \
|\widehat{\nu}|=|\nu|
 \ \Big\}.
\end{equation}
The induction step is established as before if
the set of equations as $\mu'$ and $\nu'$ varies has maximal 
rank{\footnote{Maximal rank here is
equal to the number $\mathcal{P}(|\mu|)\cdot \mathcal{P}(|\nu|)$.}} 
with respect to the unknowns \eqref{bxs2}.
Again, the maximal rank statement follows from
Lemma 8 of \cite{moop}.
\qed

\subsection{ Case $|\alpha|\geq |\lambda|$}

\subsubsection{Integration}
We will find relations which express
$\bC(\alpha|\lambda, \mu,\nu)$ in terms
of inductively treated vertices.
Let $k$ be a positive integer satisfying
$$k >  3|\alpha| + 3|\lambda|+3|\mu|+3|\nu|\ .$$
The relations will
be obtained from
vanishing stable pairs invariants of 
the nonsingular projective toric
3-fold $\bF_k\times \Pp$
in curve class
$$\beta = |\nu|\cdot [C_0^+] + |\mu|\cdot [L_0^\star]+|\lambda|\cdot
[P_{\star^+}] \  $$
following the conventions of Section \ref{wrrt}.
The virtual dimension of the associated moduli space is 
$$\text{dim}^{vir}\ P_n(\bF_k\times\Pp, \beta) > 
3|\alpha| + 3|\lambda|+3|\mu|+3|\nu|\ $$
since we easily compute
$$c_1(\bF_k\times \Pp) \cdot [C^+_0] = k+2\ , \ \ \ 
c_1(\bF_k\times \Pp) \cdot [L^\star_0] = 2, \ \ \ 
c_1(\bF_k\times \Pp) \cdot [P_{\star^+}] =2.$$

Our relations will be parameterized by three partitions 
$\lambda',\mu',\nu'$ 
which satisfy
$$|\lambda'|\leq|\lambda|-1, \ \ \ |\mu'|\leq |\mu|-1 , \ \ \
|\nu'|\leq |\nu|-1 
\ .$$
Consider the descendent insertion
\begin{equation}\label{nyr2}
\prod_{i=1}^{\ell(\alpha)} \tau_{\alpha_i}([\star^+_0]) \cdot
\prod_{i=1}^{\ell(\lambda')} \tau_{\lambda'_i}([\star^+_\infty]) 
\cdot \prod_{i=1}^{\ell(\mu')} \tau_{\mu'_i}([\star^-_0]) \cdot
\prod_{i=1}^{\ell(\nu')} \tau_{\nu'_i}([\bullet^+_0]) \cdot
\ .
\end{equation}
Since $\tau_r(\mathsf{p})$ imposes $r+2$ conditions,
the descendent insertion \eqref{nyr2} imposes at most
$$3|\alpha|+ 3|\lambda'|+ 3|\mu'|+3|\nu'|$$ 
conditions.

\begin{Proposition} The $\T$-equivariant integral 
\begin{equation*}
\int_{ [P_n(\bF_k\times \Pp,\beta)]^{vir}}
\prod_{i=1}^{\ell(\alpha)} \tau_{\alpha_i}([\star^+_0]) \cdot
\prod_{i=1}^{\ell(\lambda')} \tau_{\lambda'_i}([\star^+_\infty])\cdot  
\prod_{i=1}^{\ell(\mu')} \tau_{\mu'_i}([\star^-_0]) \cdot
\prod_{i=1}^{\ell(\nu')} \tau_{\nu'_i}([\bullet^+_0]) \cdot
\end{equation*}
vanishes for all $n$. \label{ggh223}
\end{Proposition}

\begin{proof}
The integral,
viewed as $\T$-equivariant push-forward to a point,
has dimension at least 
$$3|\alpha| + 3|\lambda|+3|\mu|+3|\nu| - 
(3|\alpha|+ 3|\lambda'|+ 3|\mu'|+3|\nu'|) >0,$$
so always vanishes.
\end{proof}

\subsubsection{Relation}
We define the $\T$-equivariant series 
\begin{multline*}
\bZ_\beta\Big(\alpha,\lambda',\mu',\nu'\Big)
=
\sum_{n}  q^n
\int_{ [P_n(\bF_k\times \Pp,\beta)]^{vir}}
\prod_{i=1}^{\ell(\alpha)} \tau_{\alpha_i}([\star^+_0]) \cdot
\prod_{i=1}^{\ell(\lambda')} \tau_{\lambda'_i}([\star^+_\infty]) \\ \cdot 
\prod_{i=1}^{\ell(\mu')} \tau_{\mu'_i}([\star^-_0]) \cdot
\prod_{i=1}^{\ell(\nu')} \tau_{\nu'_i}([\bullet^+_0]) \cdot
\end{multline*}
By Proposition \ref{ggh223}, 
the series $\bZ_\beta\Big(\alpha,\lambda',\mu',\nu'\Big)$
vanishes identically.
We will calculate the left side of
\begin{equation}\label{ttxx2}
\bZ_\beta\Big(\alpha,\lambda',\mu',\nu'\Big)=0
\end{equation}
by capped localization to obtain a relation constraining
capped descendent vertices.

Capped descendent vertices at $\T$-fixed points other than
$\star^+_0$ will have descendent partition of size less than
$|\alpha|$ and thus are known inductively to be rational.
Consider the capped 3-leg descendent vertex at $\star^+_0$.
Let $d_C$ be the associated edge degree of $C^+_0$ in
the capped localization formula.
By the geometry of $\bF_k$, the class
$$\beta - d_C[C^+_0]\in H_2(\bF_k,\mathbb{Z})$$
is not effective if $d_C>|\nu|$.
If $d_C<|\nu|$, then the vertex at $\star_0$
is known inductively to be rational since the minimal relative partition is
lower. Hence, we restrict ourselves to the principal terms where $d_C= |\nu|$.

Consider the edge degree $d_L$ of $L^\star_0$.
If $d_L<|\mu|$, then the vertex at $\star_0$
is known inductively to be rational since the middle relative partition is
lower. As before, $d_L>|\mu|$ is not permitted.
We restrict ourselves to the principal terms where $d_L= |\mu|$.

Finally, consider the edge degree $d_P$ of $P_{\star^+}$.
If $d_P<|\lambda|$, then the vertex at $\star_0$
is known inductively to be rational since the maximal relative partition is
lower. Again, $d_P>|\lambda|$ is not permitted.
We restrict ourselves to the principal terms where $d_P= |\lambda|$.

In the principal terms of the
 capped localization of
 \eqref{ttxx2}, precisely the following set of
capped 3-leg descendent vertices occur at $\star^+_0$:
\begin{equation*}
\Big\{ \ \bC(\alpha|\widehat{\lambda},\widehat{\mu},\widehat{\nu})\ \Big| \ 
|\widehat{\lambda}|=|\lambda|,\ 
|\widehat{\mu}|=|\mu|,\ |\widehat{\nu}|=|\nu| \ 
  \Big\}.
\end{equation*}
In the principal terms, the vertex
$\bC(\alpha|\widehat{\lambda},\widehat{\mu},\widehat{\nu})$ at $\star^+_0$,
is accompanied by capped edges along 
$C_0^+$, $L^\star_0$, and $P_{\star^+}$\ . 
Finally, there are
capped 1-leg descendent vertices 
$$\bC(\lambda'|\widehat{\lambda}',\emptyset,\emptyset), \ \ \
\bC(\mu'|\widehat{\mu}',\emptyset,\emptyset),\ \ \ 
\bC(\nu'|\widehat{\nu}',\emptyset,\emptyset),
$$
at $\star_{\infty}^+$, $\star_0^-$, and $\bullet_0^+$ respectively
The required maximal rank condition follows from
Proposition \ref{gbb5}. \qed

\section{Log Calabi-Yau geometry}\label{lcy}

\subsection{ Relative/descendent correspondence}
\label{jj5f}

Let $X$ be a nonsingular projective 3-fold, and let 
$S\subset X$ be a nonsingular anti-canonical $K3$ surface.
Let 
$N$
be the normal bundle of $S$ in $X$.
Let
$$S_0,S_\infty \subset \mathbf{P}(\OO_S \oplus N)\rightarrow S$$
be the sections determined by the summand $\OO_S$ and $N$
respectively. Let
$$\iota_0: S \hookrightarrow \mathbf{P}(\OO_S \oplus N)$$
be the section onto $S_0$.

Let $\mathcal{B}$ be a fixed self-dual basis of the cohomology of $S$.
Recall a Nakajima basis element in the Hilbert scheme 
$\text{Hilb}(S,n)$ is a cohomology weighted partition $\mu$ of $n$,
$$( (\mu_1,\gamma_1), \ldots, (\mu_\ell,\gamma_\ell))\ ,  \ \ \
n=\sum_{i=1}^\ell \mu_i, \ \ \ \gamma_i \in \mathcal{B}\ . $$
Such a weighted partition determines a descendent
insertion
$$\tau[\mu]=\prod_{i=1}^\ell \tau_{\mu_i-1}(\iota_{0*}(\gamma_i)) \ .$$

By $K3$ vanishing arguments explained in Section \ref{g669},
the stable pairs invariants of the relative 3-fold geometry
$\mathbf{P}(\OO_S \oplus N)/S_\infty$ are nontrivial only for
curves classes in the fibers of
\begin{equation}\label{vvgg}
\mathbf{P}(\OO_S \oplus N)/S_\infty\rightarrow S \ .
\end{equation}
Define the partition function for the relative geometry by
\begin{equation}\label{fq22}
\bZ_d^{\mathbf{P}(\OO_S \oplus N)/S_\infty}\Big(\mu,\nu\Big)
=
\sum_{n}  
\Big\langle \tau[\mu] \ \Big|\ \nu\
\Big\rangle_{\!n,d} 
^{\mathbf{P}(\OO_S \oplus N)/S_\infty}
q^n.
\end{equation}
Here, $\mu$ and $\nu$ are partitions of $d$ weighted
by $\mathcal{B}$. The  curve class
is $d$ times a fiber of \eqref{vvgg}.
By further vanishing, only the leading $q^d$ terms of \eqref{fq22}
are possibly nonzero.

\begin{Proposition}\label{dgbb5}
Let $d>0$ be an integer.
The square matrix indexed by $\mathcal{B}$-weighted partitions
of $d$ with coefficients
$$\bZ_d^{\mathbf{P}(\OO_S \oplus N)/S_\infty}
\Big(\mu,\nu\Big)$$
has maximal rank.
\end{Proposition}

\begin{proof}
We need only consider the leading $q^d$ coefficients to
prove the maximal rank statement.
By standard arguments,
dimension constraints imply the matrix is upper triangular with
respect to a suitable length ordering. The diagonal elements
are non-zero by the calculation \eqref{appp}.

More precisely, the 
lowest Euler characteristic 
moduli space of stable pairs
$$P_{d}(\mathbf{P}(\OO_S \oplus N)/S_\infty,dF)$$
in $d$ times the fiber class is canonically isomorphic to
$\text{Hilb}(S, d)$.
The $q^d$ coefficients of $\bZ_d^{\mathbf{P}(\OO_S \oplus N)/S_\infty}
\Big(\mu,\nu\Big)$ are simply the pairings{\footnote{See Lemma
of \cite{part1} for an equivariant study the same pairing for
$\text{Hilb}(\com²,d)$.}}  of
the descendents
$\tau[\mu]$ with the Nakajima basis elements $\nu$,
\begin{equation} \label{gre44}
\int_{{\text{Hilb}}(S,d)} \tau[\mu] \cdot \nu\ .
\end{equation}

The codimension of the class $\tau[\mu]$ on $\text{Hilb}(S,d)$
is 
$$d-\ell(\mu)+ \sum_i \text{codim}_\com (\gamma_i)\ . $$
Similarly the dimension of $\nu$ is
$$d-\ell(\nu)+ \sum_j \text{dim}_\com(\delta_j)\ $$
where the cohomology weights of $\nu$ are $\delta_j$.
Certainly the pairing \eqref{gre44}
vanishes unless
\begin{equation}\label{tr55}
\sum_i \text{codim}_\com (\gamma_i) = 
\sum_j \text{dim}_\com(\delta_j) + \ell(\mu) - \ell(\nu)\ .
\end{equation}
We first establish the vanishing 
\begin{equation}\label{gr66}
\int_{{\text{Hilb}}(S,2d)} \tau[\mu] \cdot \nu = 0
\end{equation}
if $\ell(\mu)> \ell(\nu)$.
The proof is by considering cycles in the symmetric product
$$\epsilon: \text{Hilb}(S,d) \rightarrow \text{Sym}^{d}(S) \ .$$
The cycle $\nu$ has image{\footnote{We are considering the
actual set theoretic image for general choices
of the classes $\delta_j$.}}
under $\epsilon$ of dimension $\sum_{i}{\text{dim}_\com(\delta_j)}$.
On the other hand, $\tau[\mu]$ is supported
on $\text{Hilb}(S,d)$  over a cycle of
codimension at least  $\sum_{i}{\text{codim}_\com(\gamma_i)}$
in $\text{Sym}^d(S)$ determined just by the incidence
conditions with the cycles $\gamma_i$.
If $\ell(\mu)>\ell(\nu)$, the latter codimension
in $\text{Sym}^d(S)$ exceeds the dimension of
$\epsilon(\nu)$ in $\text{Sym}^d(S)$
by \eqref{tr55}. The corresponding cycles in $\text{Sym}^d(S)$
can then be chosen with empty incidence --- proving the
vanishing \eqref{gr66}.

If $\ell(\mu)=\ell(\nu)$, the corresponding cycles
$\text{Sym}^d(S)$ will intersect in finitely many points.
In fact, the intersection can easily be shown to vanish
unless there is a bijection $\sigma$ satisfying
$\gamma_{\sigma(j)} = \delta_j^\vee$.
Further dimension constraints then require $\mu_{\sigma(j)}=\nu_{j}$
for nonvanishing.
Finally, the diagonal pairings are determined by  \eqref{appp}.
\end{proof}

We note Proposition \ref{dgbb5} holds for any
nonsingular projective surface $S$. The proof
does not use any special properties of $K3$ surfaces.

\subsection{Rationality}
Let $X$ be a nonsingular projective toric $3$-fold with 
an anti-canonical $K3$ divisor
$$\iota:S\hookrightarrow X \ . $$
Let $\beta \in H_2(X,\mathbb{Z})$ be a curve class, and let
$$d= \int_\beta [S] \ . $$
Let $\mu$ be a $\mathcal{B}$-weighted partition of $d$. We associate
a descendent insertion{\footnote{We consider
here descendents of the cohomology $H^*(X,\mathbb{Q})$,
not the $\mathbf{T}$-equvariant cohomology.}} to $\mu$ as before,
$$\tau[\mu]=\prod_{i=1}^\ell \tau_{\mu_i-1}(\iota_{*}(\gamma_i)) \ .$$
Consider the descendent partition function of $X$,
\begin{equation}\label{fq223}
\bZ^X_\beta\Big( \mathbf{I}\cdot \tau[\mu])
\Big)
=
\sum_{n}  
\Big\langle \mathbf{I} \cdot \tau[\mu]
\Big\rangle_{\!n,\beta} 
^{X}
q^n
\end{equation}
where I is {\em any} descendent insertion for $X$.

We may degenerate $X$ to the normal cone of $S$. Then, the 
degeneration formula expresses
$\bZ^X_\beta\Big( \mathbf{I}\cdot \tau[\mu])
\Big)$ in terms of the relative geometries
$$(X/S) \ \ \text{and} \ \ (\mathbf{P}(\OO_S\oplus N), S_\infty)\ .$$

The partition function $\bZ^X_\beta\Big( \mathbf{I}\cdot \tau[\mu])
\Big)$ is rational by Theorem \ref{nnn}.
The relative theory of $(\mathbf{P}(\OO_S\oplus N), S_\infty)$
yields rational functions in $q$ by the vanishings discussed above.
The invertibility of Proposition \ref{dgbb5} applied to
the degeneration formula inductively implies
the following result (of which Theorem \ref{fff} is a special case).

\begin{Theorem}
\label{ffff} Let 
$X$ be a nonsingular projective toric 3-fold
with an anti-canonical $K3$ section $S$. The partition function
$\mathsf{Z}^{X/S}_\beta
\Big(\mathbf{I} \ \Big| \mu \Big)$ 
is the 
Laurent expansion of a rational function in $q$.
\end{Theorem}

\subsection{$K3$ vanishing}\label{g669}
The vanishing{\footnote{
The results here were communicated to the authors by R. Thomas.}}
 of 
the stable pairs invariants of the relative 3-fold geometry
$\mathbf{P}(\OO_S \oplus N)/S_\infty$ in
all cases except for the minimal Euler characteristic in 
the fibers classes of
\begin{equation}\label{vvgg45}
\epsilon: \mathbf{P}(\OO_S \oplus N)/S_\infty\rightarrow S \ 
\end{equation}
can be seen by constructing a trivial 1-dimension
quotient of the obstruction theory, see \cite{cosec,mpt}.

We consider first the absolute geometry
$\mathbf{P}=\mathbf{P}(\OO_S \oplus N)$. To start,
fix a trivialization of the canonical bundle of the $K3$ surface $S$,
\begin{equation}\label{fvvbv}
\omega_S = \mathcal{O}_S\ .
\end{equation}
Let $T_\epsilon$ be the tangent
bundle to the fibers of $\epsilon$, and let
$$\mathcal{I}^\bullet= [\mathcal{O}_X \stackrel{s}{\rarr} F]$$
be a stable pair on $\mathbf{P}$.
 A canonical map
\begin{equation}\label{kxxns}
 T_\epsilon \rightarrow {\mathcal{E}}xt^1(\mathcal{I}^\bullet,
\mathcal{I}^\bullet)_0
\end{equation}
is obtained by infinitessimal translation along the vector field.
On the right,  ${\mathcal{E}}xt^1(\mathcal{I}^\bullet,\mathcal{I}^\bullet)_0$
is traceless sheaf Ext.
After tensoring \eqref{kxxns} 
with $\omega_\epsilon$, the dual of $T_\epsilon$, we find
\begin{equation*}
\mathcal{O}_{\mathbf{P}} \rightarrow  {\mathcal{E}}xt^1(\mathcal{I}^\bullet,
\mathcal{I}^\bullet)_0 \otimes \omega_\epsilon =
{\mathcal{E}}xt^1(\mathcal{I}^\bullet,
\mathcal{I}^\bullet \otimes \omega_{\mathbf{P}})_0\ 
\end{equation*}
since $\omega_\epsilon$ is canonically isomorphic to $\omega_{\mathbf{P}}$
using the fixed trivialization \eqref{fvvbv}.
Since all the lower ${\mathcal{E}}xt$ sheaves vanish \cite{pt},
we get a canonical map
$$
\mathcal{O}_{\mathbf{P}} \rightarrow
\mathcal{RH}om( \mathcal{I}^\bullet, \mathcal{I}^\bullet)_0[-1] \ \ \ 
\text{via} \ \ \
{\mathcal{E}}xt^1(\mathcal{I}^\bullet,
\mathcal{I}^\bullet \otimes \omega_{\mathbf{P}})_0\rarr
\mathcal{RH}om( \mathcal{I}^\bullet, \mathcal{I}^\bullet\otimes
\omega_{\mathbf{P}})_0[-1] \ . 
$$
After taking hypercohomology and applying Serre duality, we obtain
\begin{equation}\label{dccn}
\com \rarr \text{Ext}^1(\mathcal{I}^\bullet,\mathcal{I}^\bullet\otimes
\omega_{\mathbf{P}}), \ \ \ \
\text{Ext}^2(\mathcal{I}^\bullet,\mathcal{I}^\bullet) \rarr \mathbb{C}\ . 
\end{equation}
 
If the stable pair $\mathcal{I}^ \bullet$
is not of minimal Euler characteristic in
a fiber class, 
the map on the left in \eqref{dccn} is not trivial (as vertical translation then
induces a nontrivial deformation of $\mathcal{I}^ \bullet$).
Hence, the map on the right in \eqref{dccn}
provides the desired trivial quotient of the obstruction theory.

For the relative geometry $\mathbf{P}(\OO_S \oplus N)/S_\infty$,
we consider the deformation theory relative to the Artin
stack of destabilizations of the target. Then, the construction of
the trivial quotient is exactly as above.
A destabilization $\mathcal{D}$ of $\mathbf{P}(\OO_S \oplus N)/S_\infty$
maps canonically to $S$,
$$\epsilon: \mathcal{D} \rarr S\ .$$
We work with the relative dualizing sheaf $\omega_\epsilon$ and
define $T_\epsilon = \omega_\epsilon^*$. The rest is the same.

\vspace{+16 pt}
\noindent Departement Mathematik \hfill Department of Mathematics \\
\noindent ETH Z\"urich \hfill  Princeton University \\
\noindent rahul@math.ethz.ch  \hfill rahulp@math.princeton.edu \\

\vspace{+8 pt}
\noindent
Department of Mathematics\\
Princeton University\\
apixton@math.princeton.edu

\end{document}